\documentclass[11pt]{article}
\usepackage[T1]{fontenc}
\usepackage[utf8]{inputenc}
\usepackage{lmodern}
\usepackage{amsmath,amssymb,amsthm,mathtools,booktabs}
\usepackage[a4paper,margin=26mm]{geometry}
\usepackage[numbers,sort&compress]{natbib}
\usepackage[hidelinks]{hyperref}
\usepackage{microtype}
\allowdisplaybreaks[2]
\newtheorem{theorem}{Theorem}[section]
\newtheorem{lemma}[theorem]{Lemma}
\theoremstyle{definition}
\newtheorem{assumption}[theorem]{Assumption}
\theoremstyle{remark}
\newtheorem{remark}[theorem]{Remark}
\newcommand{\E}{\mathbb E}
\newcommand{\Pp}{\mathbb P}
\newcommand{\F}{\mathcal F}
\newcommand{\ucp}{\xrightarrow{\mathrm{ucp}}}
\newcommand{\st}{\xRightarrow{\mathcal L\text{-}s}}
\newcommand{\ou}{o_{\mathrm{ucp}}(1)}
\newcommand{\dd}{\,\mathrm d}
\newcommand{\ind}{\mathbf 1}

\title{A note on quarticity estimation based on spot volatility}
\author{Yi Guo\\[3pt]
\small Institute for Financial Studies, Shandong University\\
\small Jinan, 250100, China}
\date{}

\begin{document}
\maketitle
\begin{abstract}
We study an estimator of twice the integrated quarticity of a continuous It\^o semimartingale, constructed from squared returns and local estimates of spot volatility. For a shrinking estimation window, we establish a functional stable central limit theorem for a bias-corrected version of the estimator. The limiting conditional variance is $56\int_0^t c_s^4\dd s$, where $c$ denotes spot variance. The proof identifies a common martingale approximation for the three components of the estimator and determines their joint covariance structure. The volatility process may have jumps. We also compare the limiting variance with those of the realized-quarticity and quadratic spot-volatility estimators.
\end{abstract}

\section{Model and main result}
Let $(\Omega,\F,(\F_t)_{t\ge0},\Pp)$ be a filtered probability space satisfying the usual conditions, and let $W$ be a Brownian motion with respect to this filtration. Consider
\begin{equation}\label{model}
X_t=X_0+\int_0^t b_s\dd s+\int_0^t\sigma_s\dd W_s,
\qquad c_t=\sigma_t^2.
\end{equation}
The process is observed at times $j\Delta_n$, where $\Delta_n\downarrow0$. Write
\[
\Delta_j^nX=X_{j\Delta_n}-X_{(j-1)\Delta_n},\qquad
N_n(t)=\lfloor t/\Delta_n\rfloor,\qquad
Q_t=\int_0^tc_s^2\dd s,\quad V_t=2Q_t.
\]
Throughout, sums with an upper limit smaller than the lower limit are zero.

The quantity $V$ is the limiting conditional variance process in the central limit theorem for realized volatility. Under the assumptions stated below,
\[
\Delta_n^{-1/2}\left(\sum_{j=1}^{N_n(t)}(\Delta_j^nX)^2
-\int_0^t c_s\dd s\right)
\st\sqrt2\int_0^t c_s\dd B_s,
\]
where $B$ is a Brownian motion, defined on an extension of the probability space and independent of $\F$; see \citet{jacod2012}. A standard estimator of $V$ is
\begin{equation}\label{realized-quarticity}
V_n^{\mathrm{rq}}(t)=\frac{2}{3\Delta_n}
\sum_{j=1}^{N_n(t)}(\Delta_j^nX)^4.
\end{equation}
An alternative is to estimate $c$ locally and integrate a quadratic function of the resulting estimates. This approach and its efficiency properties are studied by \citet{jacodqq}; related results on spot-volatility estimation are given in \citet{al,li}.

Let $k_n$ be a sequence of positive integers. Define
\begin{equation}\label{spot-estimator}
\widehat c_i^n=\frac1{k_n\Delta_n}
\sum_{r=0}^{k_n-1}(\Delta_{i+r}^nX)^2,\qquad i\ge1.
\end{equation}
Each estimate uses the $k_n$ returns preceding $\Delta_{i+k_n}^nX$. With $Y_j^n=(\Delta_j^nX)^2/\Delta_n$, the statistic considered here is
\[
\widehat V_n(t)=\Delta_n\sum_{i=1}^{N_n(t)-k_n}
(Y_{i+k_n}^n-\widehat c_i^n)^2.
\]
Motivated by the spot-volatility construction in \citet{ait14}, we consider the bias-corrected estimator
\begin{equation}\label{estimator}
\widehat V'_n(t)=\Delta_n\sum_{i=1}^{N_n(t)-k_n}
\left[(Y_{i+k_n}^n)^2-2Y_{i+k_n}^n\widehat c_i^n
+\left(1-\frac2{k_n}\right)(\widehat c_i^n)^2\right].
\end{equation}
All returns appearing in \eqref{estimator} are observed by time $t$.

\begin{assumption}\label{assumption}
The drift $b$ is adapted, c\`adl\`ag and locally bounded. The volatility $\sigma$ is a c\`adl\`ag It\^o semimartingale of the form
\begin{align}\label{sigma-model}
\sigma_t={}&\sigma_0+\int_0^t\alpha_s\dd s
+\int_0^t\gamma_s\dd\widetilde W_s\nonumber\\
&+\int_0^t\int_E\delta(s,z)\ind_{\{|\delta(s,z)|\le1\}}(\mu-\nu)(\dd s,\dd z)
+\int_0^t\int_E\delta(s,z)\ind_{\{|\delta(s,z)|>1\}}\mu(\dd s,\dd z).
\end{align}
Here $\widetilde W$ is a Brownian motion of a finite dimension on the same filtration, $\mu$ is a Poisson random measure with compensator $\nu(\dd s,\dd z)=\dd s\,\lambda(\dd z)$, $\alpha$ is locally bounded, $\gamma$ is c\`adl\`ag and locally bounded, and $\delta$ is predictable. There are stopping times $\tau_m\uparrow\infty$ and deterministic functions $\Gamma_m:E\to[0,\infty)$ such that
\[
\int_E\Gamma_m(z)^2\lambda(\dd z)<\infty,
\qquad |\delta(s,z)|\wedge1\le\Gamma_m(z)
\quad\text{on }\{s\le\tau_m\}.
\]
No independence between $W$ and the volatility drivers is assumed.
\end{assumption}

Assumption \ref{assumption} is an $H(2)$ condition on the volatility. In particular, $c=\sigma^2$ is an It\^o semimartingale. Its continuous martingale part may be correlated with $W$, and its jump part is retained throughout the argument.

We impose the window conditions
\begin{equation}\label{window}
k_n^3\Delta_n\longrightarrow\infty,
\qquad k_n^2\Delta_n\longrightarrow0.
\end{equation}

\begin{theorem}\label{main-theorem}
Under Assumption \ref{assumption} and \eqref{window}, for every $T<\infty$,
\begin{equation}\label{main-clt}
\Delta_n^{-1/2}(\widehat V'_n-V)
\st\sqrt{56}\int_0^\cdot c_s^2\dd B'_s
\quad\text{in }D([0,T]),
\end{equation}
where the space is endowed with the Skorokhod $J_1$ topology and $B'$ is a Brownian motion independent of $\F$ on an extension of the original space.
\end{theorem}

\begin{remark}\label{bias-remark}
For constant volatility, $X=\sqrt c\,W$, one has
$\E[(\widehat c_i^n)^2]=c^2(1+2/k_n)$. Consequently,
\[
\E[\widehat V'_n(t)]
=\Delta_n(N_n(t)-k_n)_+c^2\left(2-\frac4{k_n^2}\right).
\]
The correction in \eqref{estimator} therefore leaves a bias of order $k_n^{-2}$. The condition \ref{window} makes this bias negligible at the scale of \eqref{main-clt}; the omitted initial windows are negligible because $k_n\sqrt{\Delta_n}\to0$.
\end{remark}

For comparison, define the quadratic spot-volatility estimator
\[
V_n^{\mathrm{sp}}(t)=2\left(1-\frac2{k_n}\right)\Delta_n
\sum_{i=1}^{N_n(t)-k_n+1}(\widehat c_i^n)^2.
\]
Writing $J_t=\int_0^t c_s^4\dd s$, the limiting conditional variances at scale $\Delta_n^{-1/2}$ are
\begin{center}
\begin{tabular}{lc}
\toprule
Estimator of $V_t$&Limiting conditional variance\\
\midrule
$V_n^{\mathrm{sp}}(t)$&$32J_t$\\
$V_n^{\mathrm{rq}}(t)$&$\dfrac{128}{3}J_t$\\[3pt]
$\widehat V'_n(t)$&$56J_t$\\
\bottomrule
\end{tabular}
\end{center}
The first variance is the efficiency benchmark in the model classes considered by \citet{jacodqq}. The estimator \eqref{estimator} has a larger limiting variance than either comparison estimator.

\section{Proof of the central limit theorem}
For processes $R^n$, the notation $R^n=\ou$ means that $\sup_{t\le T}|R_t^n|\to0$ in probability for each $T<\infty$. We work after localization and truncation of large volatility jumps. The localized processes agree with the original processes on events whose probabilities tend to one. On each fixed interval $[0,T+1]$, we may therefore assume that $b,\sigma,c$ and the relevant drift and quadratic-variation densities are bounded. The volatility martingales have bounded jumps and bounded compensator densities for their second and fourth jump moments. These reductions follow from Assumption \ref{assumption} by the usual localization argument; see \citet{jacod2012}. Constants below may depend on $T$ and on the localization level.

Put $t_j^n=j\Delta_n$ and use
\[
c_j^n=c_{t_{j-1}^n},\qquad Z_j^n=\frac{\Delta_j^nW}{\sqrt{\Delta_n}},
\qquad e_j^n=c_j^n\bigl((Z_j^n)^2-1\bigr).
\]
The common martingale arrays are
\begin{equation}\label{ABn}
A_t^n=\sqrt{\Delta_n}\sum_{j=1}^{N_n(t)}(c_j^n)^2\bigl((Z_j^n)^4-3\bigr),
\qquad
B_t^n=\sqrt{\Delta_n}\sum_{j=1}^{N_n(t)}c_j^ne_j^n.
\end{equation}
Their predictable brackets are understood with respect to the grid filtration $(\F_{t_j^n})_j$.

\begin{lemma}[Joint martingale limit]\label{joint-lemma}
The pair $(A^n,B^n)$ converges jointly stably to a conditionally centered continuous Gaussian martingale $(A,B)$ with
\begin{equation}\label{AB-brackets}
\langle A\rangle_t=96J_t,\qquad
\langle B\rangle_t=2J_t,\qquad
\langle A,B\rangle_t=12J_t.
\end{equation}
On a suitable extension, this pair can be represented as
\begin{equation}\label{AB-limit}
B_t=\sqrt2\int_0^tc_s^2\dd B_s^{(2)},\qquad
A_t=6B_t+\sqrt{24}\int_0^tc_s^2\dd B_s^{(4)},
\end{equation}
where $B^{(2)}$ and $B^{(4)}$ are independent Brownian motions, independent of $\F$.
\end{lemma}
\begin{proof}
For a standard normal variable $Z$,
\[
\E[(Z^4-3)^2]=96,\quad
\E[(Z^2-1)^2]=2,\quad
\E[(Z^4-3)(Z^2-1)]=12.
\]
Writing $\xi_j^n=(A_{t_j^n}^n-A_{t_{j-1}^n}^n,B_{t_j^n}^n-B_{t_{j-1}^n}^n)^\top$, conditional Gaussian moments give
\[
\sum_{j=1}^{N_n(t)}\E[\xi_j^n(\xi_j^n)^\top\mid\F_{t_{j-1}^n}]
=\Delta_n\sum_{j=1}^{N_n(t)}(c_j^n)^4
\begin{pmatrix}96&12\\12&2\end{pmatrix}
\ucp J_t\begin{pmatrix}96&12\\12&2\end{pmatrix}.
\]
Moreover,
\[
\sum_{j=1}^{N_n(T)}\E[\|\xi_j^n\|^4\mid\F_{t_{j-1}^n}]
\le C_T\Delta_n\longrightarrow0,
\]
which implies the conditional Lindeberg condition.

For a bounded martingale $N$ of the original filtration, we next verify
\begin{equation}\label{stable-orthogonality}
\sum_{j=1}^{N_n(t)}\E[\xi_j^n\Delta_j^nN\mid\F_{t_{j-1}^n}]
\ucp0.
\end{equation}
Both centered Gaussian polynomials in \eqref{ABn} are terminal values of continuous stochastic integrals with respect to $W$ on their observation interval. Their conditional products with $\Delta_j^nN$ vanish if $N$ is orthogonal to $W$. For $N=W$, they vanish by odd Gaussian moments. In general, write $N=N_0+H\cdot W+N^\perp$ by the Kunita--Watanabe decomposition. Approximate $H$ in $L^2(\Pp\otimes\dd s)$ by bounded elementary predictable integrands. For such an integrand, intervals contained in one of its deterministic time blocks contribute zero, and the finitely many intervals crossing block endpoints contribute $o_{\Pp}(1)$. Cauchy--Schwarz gives, uniformly in $n$, the error bound
\[
\E\sup_{t\le T}\left\|
\sum_{j\le N_n(t)}\E\left[\xi_j^n
\int_{t_{j-1}^n}^{t_j^n}(H_s-H_s^{(m)})\dd W_s
\,\middle|\,\F_{t_{j-1}^n}\right]\right\|
\le C_T\|H-H^{(m)}\|_{L^2(\Pp\otimes\dd s)}.
\]
This proves \eqref{stable-orthogonality}. The vector martingale stable central limit theorem, Theorem IX.7.28 of \citet{jacod2013}, yields the result. The representation \eqref{AB-limit} has the covariance matrix \eqref{AB-brackets}.
\end{proof}

\begin{lemma}[Weighted lagged products]\label{lagged-lemma}
Let $w_{n,h}\ge0$, $1\le h\le k_n$, be deterministic weights satisfying
\begin{equation}\label{weights}
\sum_{h=1}^{k_n}w_{n,h}=1,\qquad
\max_{1\le h\le k_n}w_{n,h}\le C/k_n.
\end{equation}
Define
\[
P_n^w(t)=\Delta_n\sum_{j=k_n+1}^{N_n(t)}Y_j^n
\sum_{h=1}^{k_n}w_{n,h}Y_{j-h}^n.
\]
Under Assumption \ref{assumption}, $k_n\to\infty$ and $k_n^2\Delta_n\to0$,
\begin{equation}\label{lagged-expansion}
\frac{P_n^w-Q}{\sqrt{\Delta_n}}=2B^n+\ou.
\end{equation}
\end{lemma}
\begin{proof}
Suppress $n$ in $k,\Delta,c_j,e_j,Y_j,w_h$, and write
$\bar u_j=\sum_{h=1}^kw_hu_{j-h}$ for any sequence $(u_j)$. The sum over $j$ below runs from $k+1$ to $N_n(t)$ unless otherwise specified.

\emph{Step 1: replacing squared returns by frozen innovations.}
Let $r_j=Y_j-c_j-e_j$, $\rho_j=\E[r_j\mid\F_{t_{j-1}}]$ and $m_j=r_j-\rho_j$. Lemma \ref{local-lemma} gives
\begin{equation}\label{r-estimates-main}
\E[r_j^2\mid\F_{t_{j-1}}]\le C\Delta,
\qquad a_n:=\sqrt\Delta\sum_{j\le N_n(T)}\E|\rho_j|\longrightarrow0.
\end{equation}
All fixed conditional moments of $c_j+e_j$ are bounded. Martingale orthogonality and \eqref{weights} imply
\begin{equation}\label{averaged-estimates}
\E\bar e_j^2\le C/k,\qquad
\E\bar r_j^2\le C\Delta.
\end{equation}
The normalized difference between $P_n^w$ and the same statistic with $Y_j$ replaced by $c_j+e_j$ is
\begin{equation}\label{r-expansion}
\sqrt\Delta\sum_j
\bigl[r_j(\bar c_j+\bar e_j)+(c_j+e_j)\bar r_j+r_j\bar r_j\bigr].
\end{equation}
The martingale sums involving $m_j(\bar c_j+\bar e_j)$ and $e_j\bar r_j$ have maximal second moments bounded by $C_T\Delta$. The predictable terms satisfy
\[
\E\sup_{t\le T}\left|\sqrt\Delta\sum_j\rho_j\bar c_j\right|\le Ca_n,
\qquad
\sqrt\Delta\sum_j\E|\rho_j\bar e_j|\le C_Tk^{-1/2}.
\]
Also $\sqrt\Delta\sum_j\E|r_j\bar r_j|\le C_T\sqrt\Delta$. To treat the remaining sum $\sqrt\Delta\sum_jc_j\bar r_j$, interchange the indices and write each future coefficient as $c_{l+h}=c_l+(c_{l+h}-c_l)$. The resulting sum $\sqrt\Delta\sum_lc_lr_l$ tends to zero by \eqref{r-estimates-main}; the first and last $k$ indices are negligible by Lemma \ref{boundary-lemma}. The coefficient differences have the bound
\[
\sqrt\Delta\sum_{l\le N_n(T)}\sum_{h=1}^kw_h
\E\bigl[|r_l|\,|c_{l+h}-c_l|\bigr]
\le C_T\sqrt{k\Delta}\longrightarrow0.
\]
Here the expectation is estimated by Cauchy--Schwarz and \eqref{local-moments}. These bounds also hold for the supremum over partial sums after taking absolute values. Thus \eqref{r-expansion} is $\ou$.

\emph{Step 2: the volatility and quadratic-noise terms.}
After Step 1,
\begin{align}\label{four-terms}
\frac{P_n^w(t)-Q_t}{\sqrt\Delta}
={}&\frac{\Delta\sum_jc_j\bar c_j-Q_t}{\sqrt\Delta}
+\sqrt\Delta\sum_je_j\bar c_j\nonumber\\
&+\sqrt\Delta\sum_jc_j\bar e_j
+\sqrt\Delta\sum_je_j\bar e_j+\ou.
\end{align}
Use $c_jc_{j-h}=(c_j^2+c_{j-h}^2-(c_j-c_{j-h})^2)/2$. After shifting the sums of squares, the boundary error is bounded by $Ck\Delta$, while
\[
\Delta\sum_{j=k+1}^{N_n(T)}\sum_{h=1}^kw_h
\E(c_j-c_{j-h})^2\le C_Tk\Delta.
\]
The Riemann-sum estimate \eqref{riemann} therefore makes the first term of \eqref{four-terms} $\ou$. For the fourth term, predictability of $\bar e_j$ and \eqref{averaged-estimates} give
\begin{equation}\label{product-noise}
\E\sup_{t\le T}\left|\sqrt\Delta\sum_je_j\bar e_j\right|^2
\le C\Delta\sum_{j=k+1}^{N_n(T)}\E\bar e_j^2
\le C_T/k\longrightarrow0.
\end{equation}

\emph{Step 3: the first linear term.}
Since $\bar c_j-c_j$ is $\F_{t_{j-1}}$-measurable,
\[
\E\sup_{t\le T}\left|\sqrt\Delta\sum_je_j(\bar c_j-c_j)\right|^2
\le C\Delta\sum_{j=k+1}^{N_n(T)}\E(\bar c_j-c_j)^2
\le C_Tk\Delta\longrightarrow0.
\]
The initial $k$ increments of $B^n$ are negligible, so the second term of \eqref{four-terms} is $B^n+\ou$.

\emph{Step 4: the second linear term.}
Reindexing gives
\begin{equation}\label{future-coefficient}
\sqrt\Delta\sum_jc_j\bar e_j
=B_t^n+R_n(t)+\ou,
\qquad
R_n(t)=\sqrt\Delta\sum_{j=k+1}^{N_n(t)}
\sum_{h=1}^kw_he_{j-h}(c_j-c_{j-h}).
\end{equation}
The boundary terms in this identity have absolute value at most a constant times $\sqrt\Delta$ times the sum of $|e_l|$ over the first and last $k$ indices. Lemma \ref{boundary-lemma} applies.

Write $c_t=c_0+\int_0^ta_s\dd s+M_t$, where $a$ is bounded and $M$ is a locally square-integrable martingale. Its bracket has a bounded density after localization. The drift contribution to $R_n$ satisfies
\[
\E\sup_{t\le T}|R_n^{\mathrm{drift}}(t)|
\le C\sqrt\Delta\sum_{j=k+1}^{N_n(T)}\sum_{h=1}^kw_hh\Delta\,\E|e_{j-h}|
\le C_Tk\sqrt\Delta\longrightarrow0.
\]
For the martingale part, put $d_r=M_{t_r}-M_{t_{r-1}}$ and
\[
U_r=\sum_{h=1}^kw_h\sum_{l=(r-h+1)\vee1}^{r-1}e_l.
\]
The identity $M_{t_{j-1}}-M_{t_{j-h-1}}=\sum_{r=j-h}^{j-1}d_r$ and an interchange of these finite sums yield
\begin{equation}\label{martingale-shift}
R_n^{\mathrm{mart}}(t)
=\sqrt\Delta\sum_{r=1}^{N_n(t)-1}e_rd_r
+\sqrt\Delta\sum_{r=1}^{N_n(t)-1}U_rd_r+E_n(t).
\end{equation}
Only indices within distance $2k$ of the endpoints contribute to $E_n$. Indeed, for an interior index $l$, every lag $1\le h\le k$ is included in the original sum. For an endpoint index, the missing coefficient is a weighted sum of increments of $M$ over intervals of length at most $k\Delta$. Since $M$ is bounded on $[0,T+1]$ after localization,
\[
|E_n(t)|\le C\sqrt\Delta\left(
\sum_{l=1}^{2k}|e_l|+
\sum_{l=(N_n(t)-2k)\vee1}^{N_n(t)}|e_l|\right).
\]
Thus $E_n=\ou$ by Lemma \ref{boundary-lemma}. The coefficient $U_r$ is $\F_{t_{r-1}}$-measurable. It is a deterministic linear combination of at most $k-1$ mutually orthogonal martingale differences, with coefficients in $[0,1]$, and hence $\E U_r^2\le Ck$. The second sum in \eqref{martingale-shift} consequently satisfies
\[
\E\sup_{t\le T}\left|\sqrt\Delta
\sum_{r=1}^{N_n(t)-1}U_rd_r\right|^2
\le C\Delta^2\sum_{r=1}^{N_n(T)}\E U_r^2
\le C_Tk\Delta\longrightarrow0.
\]
Lemma \ref{diagonal-lemma} treats the first sum in \eqref{martingale-shift}, including the jump martingale part of $M$. It follows that $R_n=\ou$, so the third term of \eqref{four-terms} is $B^n+\ou$. Combining the four steps proves \eqref{lagged-expansion}.
\end{proof}

\begin{lemma}[Linearizations of the estimator]\label{linearizations}
Define
\begin{align*}
F_n(t)&=\Delta_n\sum_{i=1}^{N_n(t)-k_n}(Y_{i+k_n}^n)^2,\\
T_n(t)&=\Delta_n\sum_{i=1}^{N_n(t)-k_n}Y_{i+k_n}^n\widehat c_i^n,\\
S_n(t)&=\left(1-\frac2{k_n}\right)\Delta_n
\sum_{i=1}^{N_n(t)-k_n}(\widehat c_i^n)^2.
\end{align*}
Under Assumption \ref{assumption} and \eqref{window},
\begin{equation}\label{three-linearizations}
\frac{F_n-3Q}{\sqrt{\Delta_n}}=A^n+\ou,\qquad
\frac{T_n-Q}{\sqrt{\Delta_n}}=2B^n+\ou,\qquad
\frac{S_n-Q}{\sqrt{\Delta_n}}=2B^n+\ou.
\end{equation}
\end{lemma}
\begin{proof}
For
$D_n(t)=\Delta_n\sum_{j=1}^{N_n(t)}(Y_j^n)^2$, the standard even-power variation approximation for $f(x)=x^4$ gives
\begin{equation}\label{power-approximation}
\frac{D_n-3Q}{\sqrt{\Delta_n}}=A^n+\ou;
\end{equation}
see \citet{jacod2012}. This is the martingale approximation for continuous $X$ with It\^o volatility, allowing jumps in $\sigma$; its centering is the integrated function $3c^2$. Removing the first $k_n$ terms is negligible by Lemma \ref{boundary-lemma}, applied to $(Y_j^n)^2$, which proves the first assertion. The second assertion is Lemma \ref{lagged-lemma} with $w_{n,h}=1/k_n$.

For the third assertion, set
$U_n(t)=\Delta_n\sum_{i=1}^{N_n(t)-k_n}(\widehat c_i^n)^2$.
For $k=k_n\ge2$, introduce the triangular weights
\begin{equation}\label{triangular-weights}
v_{n,h}=\frac{2(k-h)}{k(k-1)},\quad 1\le h\le k-1,
\qquad v_{n,k}=0.
\end{equation}
These weights satisfy \eqref{weights}. Expanding the square in each window gives the exact identity
\begin{equation}\label{square-identity}
U_n(t)=\frac1kD_n(t)+\left(1-\frac1k\right)P_n^v(t)+H_n(t),
\end{equation}
where
\begin{equation}\label{square-boundary}
|H_n(t)|\le C\Delta_n\sum_{j\in\mathcal E_n(t)}(Y_j^n)^2,
\qquad
\mathcal E_n(t)=\{1\le j\le N_n(t):j\le2k\text{ or }j>N_n(t)-2k\}.
\end{equation}
To verify \eqref{square-identity}, an interior diagonal term $Y_j^2$ occurs in $k$ windows and has total coefficient $1/k$. An interior pair $Y_{j-h}Y_j$ occurs in $k-h$ windows and has coefficient $2(k-h)/k^2=(1-1/k)v_{n,h}$. All remaining terms contain an endpoint index. The inequality $2|Y_jY_l|\le Y_j^2+Y_l^2$, together with the bounded row sums of the coefficients, gives \eqref{square-boundary}. Lemma \ref{boundary-lemma} implies $H_n/\sqrt{\Delta_n}=\ou$.

Combining \eqref{power-approximation}, Lemma \ref{lagged-lemma} and \eqref{square-identity} yields
\[
\frac{U_n-(1+2/k_n)Q}{\sqrt{\Delta_n}}
=\frac1{k_n}A^n+2\left(1-\frac1{k_n}\right)B^n+\ou.
\]
Since $S_n=(1-2/k_n)U_n$,
\begin{align*}
\frac{S_n-Q}{\sqrt{\Delta_n}}
={}&-\frac{4Q}{k_n^2\sqrt{\Delta_n}}\\
&+\left(1-\frac2{k_n}\right)
\left[\frac1{k_n}A^n+2\left(1-\frac1{k_n}\right)B^n\right]+\ou.
\end{align*}
The first term tends to zero uniformly on compact intervals by \eqref{window}. Lemma \ref{joint-lemma} gives tightness of $A^n$ and $B^n$, and the remaining expression is $2B^n+\ou$.
\end{proof}

\begin{proof}[Proof of Theorem \ref{main-theorem}]
By \eqref{estimator}, $\widehat V'_n=F_n-2T_n+S_n$. Lemmas \ref{joint-lemma} and \ref{linearizations} imply the joint convergence
\begin{equation}\label{joint-three}
\frac1{\sqrt{\Delta_n}}
\begin{pmatrix}F_n-3Q\\-2(T_n-Q)\\S_n-Q\end{pmatrix}
=\begin{pmatrix}A^n\\-4B^n\\2B^n\end{pmatrix}+\ou
\st\begin{pmatrix}A\\-4B\\2B\end{pmatrix}.
\end{equation}
The limiting bracket matrix is
\begin{equation}\label{full-matrix}
J_t\begin{pmatrix}96&-48&24\\-48&32&-16\\24&-16&8\end{pmatrix}.
\end{equation}
The sum of the limiting components is $A-2B$, with bracket
\[
\langle A-2B\rangle_t
=\langle A\rangle_t+4\langle B\rangle_t-4\langle A,B\rangle_t
=(96+8-48)J_t=56J_t.
\]
By \eqref{AB-limit},
\[
A_t-2B_t=4\sqrt2\int_0^tc_s^2\dd B_s^{(2)}
+\sqrt{24}\int_0^tc_s^2\dd B_s^{(4)}.
\]
The normalized linear combination of $B^{(2)}$ and $B^{(4)}$ is a Brownian motion independent of $\F$. The continuous mapping theorem and stable Slutsky convergence give \eqref{main-clt}. Finally, letting the localization level tend to infinity removes the boundedness restrictions.
\end{proof}

\section{Discussion}

The bias-corrected estimator $\widehat V'$ has asymptotic conditional variance
\[
56\int_0^t c_s^4\,\dd s
\]
under the assumptions of this paper. A central point in the proof is that the leading terms in the decomposition of $\widehat V'$ must be treated jointly. Their limiting martingales are correlated, so separate stable convergence results are not sufficient to determine the limit of their linear combination. The joint stable central limit theorem, including the limiting cross-covariations, is therefore essential.

It is natural to consider extensions to a multidimensional It\^o semimartingale. In that setting, local covariance estimators replace the scalar spot-volatility estimator, and analogous statistics may be constructed from products of matrix-valued returns and local covariance estimates. The corresponding limit theory would require an explicit calculation of the full conditional covariance structure of the leading terms, together with an appropriate bias correction. These issues are not addressed here, and no multidimensional central limit theorem is claimed.

A further direction is to consider more general functionals of high-frequency returns and local volatility estimates. For suitable functions $g$ and $f$, one may study bias-corrected versions of
\[
\frac{1}{\sqrt{\Delta_n}}
\left\{
\Delta_n
\sum_{i=1}^{\lfloor t/\Delta_n\rfloor-k_n}
g\left(
\frac{\Delta_{i+k_n}^n X}{\sqrt{\Delta_n}}
\right)
f\bigl(\widehat c(k_n)_i\bigr)
-
\int_0^t \rho_{c_s}(g)f(c_s)\,\dd s
\right\},
\]
where
\[
\rho_c(g)
=
\mathbb E\left[g(\sqrt{c}\,Z)\right],
\qquad Z\sim N(0,1).
\]
The factor $\Delta_n$ in front of the discrete sum is essential for the sum to approximate its integrated counterpart. Depending on $g$ and $f$, additional bias terms may arise from the local estimation of $c$ and from the interaction between the return and the local window. Determining the appropriate bias corrections and the resulting joint covariance structure is left for future work.

\appendix
\section{Moment and remainder estimates}
All estimates in this appendix are made under the localization described at the beginning of Section 2. For readability, write $\Delta=\Delta_n$, $t_j=j\Delta$, $\E_j[\cdot]=\E[\cdot\mid\F_{t_{j-1}}]$, and suppress $n$ on $c_j,e_j,Y_j$.

\begin{lemma}[Local moments and freezing]\label{local-lemma}
For $0\le u\le1$, the following estimates hold on each localized interval:
\begin{equation}\label{local-moments}
\E[|c_{s+u}-c_s|^2\mid\F_s]\le Cu,\qquad
\E_j|Y_j|^p+\E_j|e_j|^p\le C_p\quad(p\ge2).
\end{equation}
If $r_j=Y_j-c_j-e_j$ and $\rho_j=\E_jr_j$, then
\begin{equation}\label{freezing}
\E_jr_j^2\le C\Delta,\qquad
\sqrt\Delta\sum_{j\le N_n(T)}\E|\rho_j|\longrightarrow0.
\end{equation}
Furthermore,
\begin{equation}\label{riemann}
\frac{\Delta_n\sum_{j=1}^{N_n(t)}(c_j^n)^2-Q_t}{\sqrt{\Delta_n}}
=\ou.
\end{equation}
\end{lemma}
\begin{proof}
The first estimate follows from the bounded drift and bracket densities of $c$. The conditional moments of $Y_j$ follow from the Burkholder--Davis--Gundy inequality applied to \eqref{model}; those of $e_j$ follow from conditional Gaussian moments.

Fix $a=t_{j-1}$ and let $\widetilde X_s=X_s-X_a$ and $\widetilde W_s=W_s-W_a$. The same inequalities and It\^o's product formula give, for $a\le s\le a+\Delta$,
\begin{equation}\label{mixed-moments}
\E[|\widetilde X_s-\sigma_a\widetilde W_s|^2\mid\F_a]\le C(s-a)^2,
\qquad
\E[\widetilde W_s^2(\sigma_s-\sigma_a)^2\mid\F_a]\le C(s-a)^2.
\end{equation}
For the second bound, apply the product formula to $\widetilde W^2(\sigma-\sigma_a)^2$. The second-moment densities of the volatility martingale are bounded, its continuous covariation density with $W$ is bounded, and the drift is bounded. Taking conditional expectations bounds each drift contribution by $C(s-a)$ before time integration. The jump contribution uses only the compensator of $(\Delta\sigma)^2$, so this argument includes volatility jumps.

Another application of It\^o's formula yields
\begin{align*}
\widetilde X_{a+\Delta}^2-c_a\widetilde W_{a+\Delta}^2
={}&2\int_a^{a+\Delta}(\widetilde X_s\sigma_s-c_a\widetilde W_s)\dd W_s\\
&+\int_a^{a+\Delta}(c_s-c_a)\dd s
+2\int_a^{a+\Delta}\widetilde X_sb_s\dd s.
\end{align*}
By \eqref{mixed-moments}, the conditional second moment of the stochastic integral is at most $C\Delta^3$. Cauchy--Schwarz and \eqref{local-moments} give the same bound for the finite-variation terms. Division by $\Delta^2$ proves $\E_jr_j^2\le C\Delta$.

For the conditional mean, put
\[
\kappa_j=\left(\E_j\sup_{s\in[t_{j-1},t_j]}|b_s-b_{t_{j-1}}|^2\right)^{1/2}.
\]
The conditional expectation of $c_s-c_a$ is bounded by $C(s-a)$. Split $b_s=b_a+(b_s-b_a)$ in the last integral and use $|\E[\widetilde X_s\mid\F_a]|\le C(s-a)$. It follows that
\[
|\rho_j|\le C\Delta+C\sqrt\Delta\,\kappa_j.
\]
Since $b$ is bounded and c\`adl\`ag after localization, dominated convergence for its grid oscillations gives $\Delta\sum_{j\le N_n(T)}\E\kappa_j\to0$. This proves the second assertion of \eqref{freezing}. The same drift-oscillation estimate is used in Lemmas 4.1--4.2 of \citet{jacodqq}.

Finally, $c^2$ is an It\^o semimartingale with locally bounded drift and bracket densities. Stochastic Fubini represents the error of its left Riemann sum, at grid endpoints, as a finite-variation integral and a martingale integral with deterministic integrands bounded by $\Delta$. Their maximal sizes are respectively $O(\Delta)$ and $O_{\Pp}(\Delta)$. The last partial interval has length at most $\Delta$ and bounded integrand. This proves \eqref{riemann}.
\end{proof}

\begin{lemma}[Moving boundary windows]\label{boundary-lemma}
Suppose $Z_j^n\ge0$ is $\F_{t_j^n}$-measurable and
$\E[(Z_j^n)^4\mid\F_{t_{j-1}^n}]\le C$. If $r_n\le C_0k_n$ and $k_n^2\Delta_n\to0$, then
\begin{equation}\label{moving-boundary}
\sqrt{\Delta_n}\max_{m\le N_n(T)}
\sum_{j=(m-r_n+1)\vee1}^{m}Z_j^n
\xrightarrow{\Pp}0.
\end{equation}
The same holds for the initial $r_n$ terms.
\end{lemma}
\begin{proof}
Write $Z_j^n=\mu_j^n+\zeta_j^n$, where $\mu_j^n$ is its conditional mean at $t_{j-1}^n$. The predictable contribution is bounded by $Cr_n\sqrt{\Delta_n}\to0$. The variables $\zeta_j^n$ are martingale differences with bounded conditional fourth moments. For any deterministic block of at most $r_n$ consecutive indices, the martingale moment inequality gives
$\E|\sum\zeta_j^n|^4\le Cr_n^2$. A union bound over at most $N_n(T)$ blocks yields
\[
\Pp\left(\sqrt{\Delta_n}\max_{m\le N_n(T)}
\left|\sum_{j=(m-r_n+1)\vee1}^{m}\zeta_j^n\right|>\varepsilon\right)
\le C_T\varepsilon^{-4}r_n^2\Delta_n\longrightarrow0.
\]
This proves the result.
\end{proof}

\begin{lemma}[Products on the same observation interval]\label{diagonal-lemma}
Let $M=M^c+M^d$ be a localized martingale whose continuous bracket density is bounded and whose purely discontinuous part has bounded second- and fourth-moment compensator densities. Then
\begin{equation}\label{diagonal-limit}
D_n^M(t):=\sqrt{\Delta_n}\sum_{j=1}^{N_n(t)}
e_j^n\Delta_j^nM\ucp0.
\end{equation}
For the purely discontinuous part one has the stronger estimate
\begin{equation}\label{jump-maximal}
\E\sup_{t\le T}|D_n^{M^d}(t)|^2\le C_T\sqrt{\Delta_n}.
\end{equation}
\end{lemma}
\begin{proof}
First consider $M^d$. The identity
\[
e_j^n=\frac{2c_j^n}{\Delta_n}
\int_{t_{j-1}^n}^{t_j^n}(W_s-W_{t_{j-1}^n})\dd W_s
\]
represents $e_j^n$ as the terminal value of a continuous martingale on its observation interval. This martingale is strongly orthogonal to $M^d$, and therefore
\begin{equation}\label{jump-conditional}
\E[e_j^n\Delta_j^nM^d\mid\F_{t_{j-1}^n}]=0.
\end{equation}
Thus $D_n^{M^d}$ is a grid martingale. The martingale moment inequalities for the jump part give $\E|\Delta_j^nM^d|^4\le C\Delta_n$, while $\E|e_j^n|^4\le C$. Doob's inequality and Cauchy--Schwarz now imply
\begin{align*}
\E\sup_{t\le T}|D_n^{M^d}(t)|^2
&\le4\Delta_n\sum_{j=1}^{N_n(T)}
\E[(e_j^n)^2(\Delta_j^nM^d)^2]\\
&\le4\Delta_n\sum_{j=1}^{N_n(T)}
\bigl(\E|e_j^n|^4\bigr)^{1/2}
\bigl(\E|\Delta_j^nM^d|^4\bigr)^{1/2}
\le C_T\sqrt{\Delta_n}.
\end{align*}
This controls the unconditioned random sum in \eqref{diagonal-limit}.

For $M^c$, decompose the sum into its predictable and martingale parts:
\begin{align}\label{continuous-diagonal-split}
D_n^{M^c}(t)={}&\sqrt{\Delta_n}\sum_{j\le N_n(t)}
\E_j[e_j^n\Delta_j^nM^c]\nonumber\\
&+\sqrt{\Delta_n}\sum_{j\le N_n(t)}
\left(e_j^n\Delta_j^nM^c-\E_j[e_j^n\Delta_j^nM^c]\right).
\end{align}
Since $\E|\Delta_j^nM^c|^4\le C\Delta_n^2$, the second term has maximal second moment at most $C_T\Delta_n$.

To treat the first term, write $M^c=\beta\cdot W+L$, where $L$ is orthogonal to $W$ and $\beta$ is predictable and bounded after localization. The $L$ contribution is zero. Approximate $\beta$ in $L^2(\Pp\otimes\dd s)$ by bounded elementary predictable processes $\beta^{(m)}$. On each observation interval contained in a time block of $\beta^{(m)}$, the conditional expectation is zero because
$\E_j[e_j^n\Delta_j^nW]=0$. The finitely many intervals crossing block endpoints have total $L^1$ contribution $O_m(\Delta_n)$ after multiplication by $\sqrt{\Delta_n}$. For the approximation error,
\begin{align*}
&\E\sup_{t\le T}\left|
\sqrt{\Delta_n}\sum_{j\le N_n(t)}
\E_j\left[e_j^n\int_{t_{j-1}^n}^{t_j^n}
(\beta_s-\beta_s^{(m)})\dd W_s\right]\right|\\
&\hspace{25mm}\le C_T\|\beta-\beta^{(m)}\|_{L^2(\Pp\otimes\dd s)}.
\end{align*}
Letting $n\to\infty$ and then $m\to\infty$ proves that the predictable part vanishes. Both terms in \eqref{continuous-diagonal-split} therefore converge to zero uniformly in probability, completing the proof.
\end{proof}

\end{document}